\documentclass[11pt, reqno]{amsart}
\usepackage{amsmath, amsthm, amssymb, amscd, mathrsfs, amsfonts,
graphicx, tikz, tikz-cd, pgf, hyperref, mathtools, scalerel, float,
subfigure, multicol, comment}
\usepackage[margin=3cm]{geometry}
\usepackage[utf8]{inputenc}
\usepackage[T1]{fontenc}
\usepackage[font=small,labelfont=bf]{caption}
\makeatletter \renewcommand{\fnum@figure}{Fig. \thefigure} \makeatother
\newtheorem{theorem}{Theorem}[section]
\newtheorem{corollary}[theorem]{Corollary}

\newtheorem{examples}[theorem]{Examples}

\newtheorem{lemma}[theorem]{Lemma}

\newtheorem{proposition}[theorem]{Proposition}
\newtheorem{remark}[theorem]{Remark}

\numberwithin{equation}{section}
\numberwithin{figure}{section}
\newcommand{\z} {\mathbb{Z}}
 \renewcommand{\r} {\mathbb{R}}

\newcommand{\La}{\Lambda}

\newcommand{\ra}{\rightarrow}

\begin{document}
\title[Some computations in string topology ]{Some computations in string topology}
\author{Arun Maiti}
\address{Arun Maiti, Department of Mathematics, Indian Institute of Science,
Bangalore 560 012}
\email{arunmaiti@iisc.ac.in}
\keywords{String topology, Hochschild homology, Geodesics}
\subjclass[2010]{ 05B45 }
\maketitle

\begin{abstract}
 In this paper, we discuss Hochschild chain models for some of the string topology operations. We use these models to simplify the proofs and computations of some of the results in string topology. Along the way we also make some new observations. We further discuss how nonnilpotent local level homology classes with respect to the Chas-Sullivan and the Goresky-Hingston product detect closed geodesics with optimal index growth rates.
\end{abstract}

\section{Introduction}

String topology is the study of algebraic structures on homology of free loop spaces, pioneered by M. Chas and S. Sullivan \cite{CS99}. These structures are usually referred to as the string topology operations. For a compact Riemannian manifold $M$ of dimension $n$, the free loop space of $M$, $\Lambda(M)$ is the space of all the continuous maps from the circle to the manifold $M$. In geometric settings, however, it is more convenient to work with the completion of all piecewise smooth curves from the circle to the manifold, 
 \[\Lambda(M):=H^1 ( S^1, M)\cong \{ H^{1} ( \mathbb{R}/\mathbb{Z}, M))\}.\]
 It admits the structure of a Hilbert manifold \cite{Kl78}. We denote $\Lambda(M)$ simply by $\Lambda$ when there is no confusion about the underlying manifold $M$. Among the string operations, here we consider the Chas-Sullivan product, the degenerate coproducts and the Goresky-Hingston product. 
\newline
Since their discovery, many different interpretations of the string operations have been given by various authors, e.g., see \cite{CJ02} for a homotopy theoretic interpretation, \cite{CS08} for a Morse theoretic and \cite{M04} for a De Rahm theoretic interpretation. Here we discuss Hochschild chain models for the above mentioned string operations. 
 \newline
 We apply the algebraic models to establish certain properties of the operations. In particular, we determine the ring structure for monogenic manfiolds with respects tot he Goresky-Hingston product. We further use a mix of geometric and algebraic methods to study the properties of the Goresky-Hingston product and compute it for the sphere and describe some generators for the complex projective spaces. 
\newline
A closed geodesic is an element of $\Lambda$ that is also a geodesic. In order to investigate the existence and properties of closed geodesics, geometers often like to view them as critical points of energy functional on the free loop space. The energy and length of an element $\alpha \in \Lambda$ are defined respectively by 
\begin{equation}
 E(\alpha)= \int_{\mathbb{R}/\mathbb{Z}} g(\alpha'(t), \alpha'(t)) dt \ \ 
\mbox{and} \ \ L(\alpha)= \int_{\mathbb{R}/\mathbb{Z}} \sqrt{g(\alpha'(t), \alpha'(t))} dt. 
\end{equation}
 There is a natural $O(2)=SO(2) \times \{\pm 1\}$ action on the free loop space defined by 
\[O(2) \times \Lambda \ra \Lambda, (e^{2 \pi it}, \gamma) \ra \beta, \beta(s)= \gamma(s+t); (-1, \gamma) \ra \beta, \beta(s)= \gamma(-s).\] 
So, each closed geodesics $\gamma$ corresponds to a $O(2)$ family of closed geodesics, called the \textit{orbit} of $\gamma$ and denoted 
by $\bar{\gamma}$. 
Each closed geodesic $\gamma$ is also associated with an infinite family of closed geodesics, by 
taking iterates of $\gamma$, $\gamma^m$ for $m=\pm 1, \pm 2, \cdots $. So, the circle action and 
the iterates of a geodesic produces geometrically the same object, meaning, their images are the same
subset of $M$. 
\newline
 A closed geodesic is said to be \textit{isolated} if it represents an isolated point in the equivariant free loop space $\Lambda/O(2)$. It was shown in \cite{GH09} that the products extends to level homologies $H_*(\La^{\leq a}, \La^{<a})$ and local level homologies of isolated closed geodesics, $H_*(\La^{< a} \cup \bar{\gamma}, 
\La^{< a})$, when the levels are defined with respect to the square root of the energy functional 
$F=\sqrt E$, i.e., 
 
\[\La^{\leq a}= \{\alpha \in \Lambda | \ F(\alpha)\leq a \} \hspace{0.5 cm} \mathrm{and} \hspace{0.5 
cm}
\La^{< a}= \{\alpha \in \Lambda | \ F(\alpha)< a\}. \]

The critical points of $F$ are the same as of $E$. The index of a closed geodesic 
$\gamma$ is the dimension of a maximal subspace of the tangent space $T_{\gamma}(\La)$ at $ 
\gamma$ of $\La$ on which the Hessian $d^2F( \gamma)$ is negative definite, and the nullity of $
\gamma$ is $ \dim T^0_{\gamma}(\La)-1$, where $T^0_{\gamma}$ is the null space of the Hessian 
$d^2F(\gamma)$. The $-1$ is incorporated to account for the fact that every non-constant closed 
geodesic $\gamma$ occurs in an $S^1$ orbit of closed geodesics. 
A closed geodesic is said to be \textit{nondegenerate} if its nullity is zero.
 \newline
 The behaviour of the local level homologies with respect to the string products in the nondegenerate case were discussed in \cite{GH09}. In \S \ref{applications}, we discuss the case of isolated but possibly degenerate closed geodesics. We prove in Theorem \ref{second} that if the local level homology class of such geodesic is nonnilpotent with respect to the Chas-Sullivan product then $\mathrm{index}+ \mathrm{nullity} $ of the iterates of the geodesic has minimal possible growth rate, as stipulated by Bott's index estimate \ref{indexgrowth}. In a similar way we prove in Theorem \ref{third} that if the local level cohomology class of such geodesic is nonnilpotent with respect to the Goresky-Hingston product then $\mathrm{index}$ of the iterates of the geodesic has maximal possible growth rate.
 \newline
 The problem of showing the existence of infinitely many geometrically distinct closed geodesics on any Riemannian manifold is usually referred to as the closed geodesic problem. Here closed 
geodesics are always understood to be non-constant. In \S \ref{applications}, we review some of the important results on this problem. The main approach to resolve this problem has been to use topological complexity of the free loop space to force the existence of critical points of the energy 
functional. Topological complexity of a space often reflects on its homology groups (for instance, on 
the growth of Betti numbers) and algebraic structures on them. So, string topology is expected to shed some light on this problem. Indeed, we shall see in \S \ref{applications} that the Theorems
\ref{second} and \ref{third} provide criteria for the existence of infinitely many geometrically distinct closed geodesics.

\section{String topology operations }
In this section, we recall the definitions of the string operations following the original description given by M. Chas and D. Sullivan in \cite{CS99, Su03}. \newline
We have the standard evaluation mappings defined by
 \[ ev_s: \Lambda(M) \rightarrow M, \ ev_s(\alpha)= \alpha(s).\]

Recall that for an embedding $V_1 \hookrightarrow V_2$ of co-dimension $m$, where $V_1$ and $V_2$ finite dimensional manifolds, the \textit{Gysin map} is the composition

\[ H_{p}(V_2) \rightarrow H_{p}(V_2, V_2 -V_1) \cong H_{p}(N, N -V_1) \xrightarrow{\cap \mu} H_{p-m}(V_1),\]

where $N$ is a tubular neighbourhood of
$V_1$ in $V_2$, and $\mu$ is the Thom class of the normal bundle $N \rightarrow V_1 $. \newline
For $t \in (0,1)$, let \[\Theta_t =\{\alpha \in \Lambda | \alpha(0)=\alpha(t)\}, \mathrm{\ and \ the \ map} \quad cut_t: \Theta_{t} \rightarrow \Lambda
\times \Lambda \] cuts a loop at time $t$, and \[\Theta =\{(\alpha, \beta) \in \Lambda \times \Lambda |
\alpha(0)=\beta(0)\}, \mathrm{\ and \ the \ map} \quad concat: \Theta \rightarrow \Lambda \] denotes the concatenation. In \cite{Su03}, the authors showed that the Gysin map can also be defined for the embeddings \[\Theta_t \hookrightarrow \Lambda \
(\mathrm{codim} \ n) \ \mathrm{and}\ \ \Theta \hookrightarrow \Lambda \times \Lambda \ (\mathrm{codim} \ n).\]
Then the maps
\[ \Lambda \times \Lambda \xhookleftarrow{ }\Theta \xrightarrow{concat} \Lambda,\] induces the \textit{Chas-Sullivan product}, denoted by $\bullet$,

\begin{equation}
 H_{i}(\Lambda) \times H_{j}(\Lambda) \xrightarrow{EZ} H_{i+j}(\Lambda \times \Lambda)
\xrightarrow{\mathrm{Gysin \ map}} H_{i+j-n}(\Theta) \xrightarrow{concat_*} H_{i+j-n}(\Lambda),
 \end{equation}
 where EZ is the Eilenberg-Zilber map. We also have the maps
 \[ \Lambda \xhookleftarrow{\mathrm{codim} \ n } \Theta_t \xrightarrow{cut_{t}} \Lambda \times \Lambda \quad
\mbox{for} \ t \in (0, 1), \]
or equivalently, the maps
\begin{equation}\label{cut} \Lambda \xrightarrow{\varGamma_t} \Lambda \xhookleftarrow{\mathrm{codim} \ n }
\Theta_{\frac{1}{2}} \xrightarrow{cut_{\frac{1}{2}}} \Lambda \times \Lambda, \quad \mbox{for} \ t \in [0, 1],
\end{equation}

where $\varGamma_t: \Lambda \rightarrow \Lambda $ is given by
\begin{equation}
 \varGamma_{t}(\gamma)(s)=
 \begin{cases}
 \gamma(\frac{1}{2}st), & \text{for}\ 0 \leq s \leq \frac{1}{2} \\
 \gamma(2t(1-s)+ 2(s-\frac{1}{2}), & \text{for}\ \frac{1}{2} \leq s \leq 1
 \end{cases}
\end{equation}
Note that $cut_{\frac{1}{2}} \circ \varGamma_t = cut_t$ on $\Theta_t$. So the above maps induce the chain maps

\begin{equation}
 P_t: C^{i}(\Lambda \times \Lambda) \xrightarrow{(cut_{\frac{1}{2}})^{\#} } C^{i}(\Theta_{\frac{1}{2}})
\xrightarrow{Gysin \ map} C^{i+n}(\Lambda) \xrightarrow{ \varGamma_{t}^{\#}} C^{i+n}(\Lambda).
\label{Pt}
\end{equation}

$P_t$ induces products on cohomology, also known as \textit{degenerate string product}

\begin{equation}\label{degenerate product}\vee_t: H^{i}(\Lambda) \times H^{j}(\Lambda) \xrightarrow{\times} H^{i+j}(\Lambda \times \Lambda) \xrightarrow{(P_t)_*} H^{i+j+n}(\Lambda),
\end{equation}
where $\times$ denote the cross product.
\newline
 We see that for any two values $t_1, t_2 \in [0, 1]$ the chain maps $P_{t_1}$ and $P_{t_2}$ are homotopic. A homotopy between $ P_{t_1}$ and $P_{t_2}$ can be defined by varying $t$ between $t_1$ and $t_2$. Therefore, for all $t \in [0, 1]$, $\vee_t$ define the same product on cohomology. 
 \newline
 Let $P:C^{i}(\Lambda \times \Lambda) \times C^{i}(\Lambda \times \Lambda)$ be the chain homotopy such that $P \delta + \delta P= P_0-P_1$. However, $P_0$, $P_1$ vanishes on $C^*(\Lambda, \Lambda_0)$ (see Remark \ref{degeneracy}), cochains relative to constant loops. Thus $P$ induces a product on cohomology, denoted by $\circledast$,
 \[H^i(\Lambda, \Lambda_0) \times H^ j(\Lambda,
\Lambda_0) \xrightarrow{\circledast} H^{ i+j+n-1}(\Lambda,
\Lambda_0),\]
 known as the \textit{Goresky-Hingston product.}
 \newline 
Let us denote by $ev^{\Lambda}_0$ the composition $ \Lambda \xrightarrow{ev_0} M \hookrightarrow \Lambda$. We have the following alternative description of the products $\vee_0$ and $\vee_1$.
 
\begin{lemma} \label{morse}
The homomorphisms $\vee_0$ and $\vee_1 $ can equivalently be described as the composition
\begin{equation}
 H^{i}(\Lambda) \times H^{j}(\Lambda) \xrightarrow{\times} H^{i+j}(\Lambda \times \Lambda) \xrightarrow{} H^{i+j}(\Theta) \xrightarrow{\tau} H^{i+j+n}(\Lambda \times \Lambda) \xrightarrow{ (ev^{\Lambda}_0 \times id)^* } H^{i+j+n}(\Lambda)
\label{P0}
\end{equation}
and
\begin{equation}
 H^{i}(\Lambda) \times H^{j}(\Lambda) \xrightarrow{\times} H^{i+j}(\Lambda \times \Lambda) \xrightarrow{} H^{i+j}(\Theta) \xrightarrow{\tau} H^{i+j+n}(\Lambda \times \Lambda) \xrightarrow{ ( id \times ev^{\Lambda}_0)^* } H^{i+j+n}(\Lambda)
\label{P1}
\end{equation}
respectively, where $\tau$ is the Gysin map for the embedding $ \Theta \hookrightarrow \Lambda \times \Lambda $.
\end{lemma}

\begin{proof}
Let $E$ and $\bar{E}$ be the normal bundles of the embeddings $ \Theta_{\frac{1}{2}} \hookrightarrow \Lambda$ and $ \Theta \hookrightarrow \Lambda \times \Lambda$ respectively, obtained by pulling back the normal bundle of the embedding $ M \hookrightarrow M\times M$. Since $cut_{\frac{1}{2}}: \Theta_{\frac{1}{2}} \rightarrow \Theta$ is a homeomorphism, therefore, there is a bundle isomorphism $h: E \rightarrow \bar{E}$. Let $\Psi_{\Lambda}: E \rightarrow N \subset \Lambda$ and $\Psi_{\Lambda \times \Lambda}: \bar{E} \rightarrow \bar{N} \subset \Lambda \times \Lambda $ be the corresponding tubular neighbourhoods, then $h$ restricts to a homeomorphism from $N$ to $\bar{N}$, and the following diagram commutes:
\begin{equation} \label{tubular}
 \begin{CD}
 \Theta_{\frac{1}{2}}@> i >>N@< \varGamma_0 << \Lambda \\
 @VV cut_{\frac{1}{2}} V @VVh V @VV id V\\
 \Theta @> \bar{i} >>\bar{N}@< (ev^{\Lambda}_0 \times id) << \Lambda
 \end{CD}
 \end{equation}
where, $i$ and $\bar{i}$ are inclusions as zero section. Let $\pi$ denote the Gysin map for the embedding $ \Theta \hookrightarrow \Lambda \times \Lambda $. Then, we see that the homomorphism $ \varGamma^*_0 \circ \pi $ is nothing but the homomorphism $ (ev^{\Lambda}_0 \times id)^* \circ \tau$. Similarly, the homomorphisms $\varGamma^*_0 \circ \pi $ and $ (id \times ev^{\Lambda}_0 )^* \circ \tau$ are the same. Thus the lemma follows. 
\end{proof}
\begin{remark}\label{degeneracy}
Since $ev_0$ and $ev_1$ define trivial maps on $C^ *(\Lambda, \Lambda_0)$, therefore it is easy to see from the above lemma that the chain maps $P_0$ and $P_1$ vanishes on $C^ *(\Lambda, \Lambda_0)$, consequently, so do the homomorphisms $\vee_0$ and $\vee_1$ on $H^ *(\Lambda, \Lambda_0)$.
\end{remark}
\section{Hochschild chain models}\label{algebriacmodel}
 In this section, we will discuss certain algebraic models for some of the string operations. Our model will rely on a cosimplicial model of the free loop space produced by J. Jones in \cite{JJ87} which we recall bellow.
\newline
\subsection{Hochschild homology}\label{hochschildchain}
 $A= k \oplus \bar{A} $ is an augmented unital differential associative k-algebra over a commutative ring $k$. The Hochschild chain complex $( CH_*(A), d)$ of a commutative differential graded algebra $A$ is defined to be $CH_*(A):= \oplus_{l \geq 0} A ^ {\otimes {l+1}} $ with differential $d= d_0 + d_1$ given by
 \[ d_0(a_0, a_1,\cdots, a_l])= \sum_{i=0}^{l}(-1)^{i} (a_0, a_1, \cdots, d_A(a_i), \cdots a_l ] \]
 \begin{equation}
 \begin{split}
 d_1 (a_0, a_1, a_2, \cdots, a_l) = \sum_{i=0}^{l-1} (-1)^{i}(a_0, a_1, \cdots, a_i a_{i+1},\cdots, a_{l}) \\+ (-1)^{l} ( a_l a_0, a_1, \cdots, a_{l-1})
 \end{split}
 \end{equation}
 $( CH_*(A), d)$ defines a chain complex and its homology $HH_*(A):= \mbox{ker}(d)/ \mbox{im} (d)$ is known as Hochschild homology of A.
 \newline
 One can write Hochschild complex $CH_*(A):= (A \otimes T(s\bar{A}), d)$, where $T(s\bar{A})$ denotes the free coalgebra generated by the graded vector space $s\bar{A}$ with $\bar{A}=\{A^i\}_{i\geq 1}$ and $(s\bar{A})^i=\bar{A}^{i+1}$. $\bar{A}$ is the kernel of the augmentation $\epsilon: A \rightarrow k$.
 We also have the relative Hochschild chain complex of $A$ defined by \[\tilde{CH}_*(A) = \oplus_{n\geq 1} A \otimes {\bar{A}}^{\otimes n}\] equipped with the Hochschild differential. There is an exact sequence

 \[ 0 \longrightarrow (A, d_A) \longrightarrow CH_*(A) \longrightarrow \tilde{CH}_*(A).\]

 The homology of $\tilde{CH}_*(A)$ is denoted by $\Tilde{HH}_*(A)$. 
 \newline
 \subsection{Jones Theorem}
We briefly recall a cosimplicial model for the free loop space and its relationship with the Hochschild homology established by J. Jones in \cite{JJ87}. 
Let $ \Delta^k=\{(t_1, t_2, \cdots, t_k)$ $\mid 0 \leq t_1 \leq \cdots $ $\leq t_k \leq 1 \} $ denote the standard $k$ -simplex. 
We define maps
 \[\mathrm{ev}_k: \Delta^k \times \Lambda \rightarrow M^{k+1}\]
 \[( t_1, \cdots, t_k, \alpha) \rightarrow (\alpha(0), \alpha(t_1), \cdots, \alpha(t_k)).\]

Let $\int_k $ be the composition
 \begin{equation}\label{jmorphi} C^*(M)^{\otimes(k+1)} \xrightarrow{Ev_k^*} C^*(\Delta^k \times \Lambda) \xrightarrow{./ [\Delta^k]} C^{*-k}(\Lambda),
 \end{equation}
 where $./ [\Delta^k]$ is the slant product by the canonical $k$-chain in chains of $\Delta^k$. Then Jones proved the following.
 \begin{theorem}\label{jones} For simply connected $M$, the homomorphism \[\int: CH_*(C^*(M)) \rightarrow C^*(\Lambda(M))\] induced by the homomorphisms $\int_k $ is a chain homotopy equivalence. It therefore induces an isomorphism $H_*(C^*(M))\cong H^*(\Lambda(M))$.
\end{theorem}
 Hochschild cohmology is defined by dualising the Hochschild chain complex identification
 $\mbox{Hom} (A^{\otimes i+1}, k) \cong \mbox{Hom}(A^{\otimes i}; A)$. We also have an isomorphism \[ HH^*(C^*(M)) \cong H_*(\Lambda).\]

\subsection{A chain model for the Chas-Sullivan product}
In [CJ02], R. Cohen and J. Jones proved the following result.
\begin{theorem}[\cite{CJ02}]For simply connected $M$, the isomorphism $\int$ is an isomorphism of graded rings $HH^*(C^*(M))\cong H_*(\Lambda(M))$ with the Chas-Sullivan product on $H_*(\Lambda(M )$ and the cup product on $HH^*(C^*(M))$.
\end{theorem}

\subsection{Chain models for the string coproducts and the Goresky-Hingston product}

It turns out that singular cochain algebra is not flexible enough to realize the string coproducts and the Goresky-Hingston product on Hochschild chain complex. However, it is possible to take a way around this problem.
\newline
First we recall that a morphism in some category of complexes is a \textit{quasi-isomorphism} if it induces an isomorphism in homology. Two objects are \textit{quasi-isomorphic} if they are related by a finite sequence of quasi-isomorphisms. It is known that if two CDGA, $A_1$ and $A_2$ are quasi-isomorphic then $HH_*(A_1) \cong HH_*(A_2)$, a fact that can be for instance deduced from Morita invariance of Hochschild homology. In light of these facts the following theorem from \cite{LS} crucial in constructing chain models for string operations in question.
\begin{theorem}[Poincaré Duality model]\label{ls} Let $M$ be simply connected
 $C^*(M)$ be singular cochain algebra over $\mathbb{Q}$. Then there is a CDGA $A$ over $\mathbb{Q}$ satisfying \newline
1) $A$ is quasi-isomorphic to $C^*(M)$, \newline
2) $A$ satisfies Poincaré duality in dimension $n$.
\end{theorem}
If $A$ be a CDGA satisfying Poincaré duality in dimension $n$ then $A^k \cong A^{n-k}$ and $A^i=0$ for $i \geq n$. Since $A$ is of finite type i.e. finitely generated $\mathbb{Q}$ module, this implies that $A$ is finite dimensional, so it admits a homogeneous finite basis $\{ a_i\}_{1 \leq i \leq N}$. If $\{a_i\}_{1 \leq i \leq N}$ is a homogeneous basis of $A$ then there is a unique basis $\{a_i^*\}_{1 \leq i \leq N}$ of $A$ characterised by the equations
\[\langle a_i, a_i^* \rangle= \delta_{ij},\]
where $ \delta_{ij}$ is the Kronecker symbol, is called the Poincaré dual basis of $\{a_i\}_{1 \leq i \leq N}$. We note that a Poincaré dual basis of $A$ is also homogeneous with $\deg(a^*_i)=n - deg(a_i)$.
\newline
In \cite{LS04} the authors shows that diagonal class exists for any CDGA satisfying Poincaré duality.
 \begin{proposition}\label{ls1} Let $A$ be a CDGA satisfying Poincaré duality in dimension $n$. Let $ \{a_i\}_{1 \leq i \leq N}$ be homogenous basis of $A$ and let $ \{a^*_i\}_{1 \leq i \leq N}$
 be its Poincaré dual basis. Then the element
 \[ \mu= \sum_{i=1}^N (-1)^{deg(a_i)} a_i \otimes a^*_i \in (A \otimes A)^n\]
 does not depend on the choice of the basis $ \{a_i\}_{1 \leq i \leq N}$.
 \end{proposition}
 
 \begin{proposition} \label{ls2}Let $A$ be a CDGA satisfying Poincaré duality in dimension $n$, and let $\mu \in (A \otimes A)^n$ be its diagonal class. Then the map
 \[ A \rightarrow A \otimes A, a \mapsto \mu.(1 \otimes a) \]
 is a homomorphism of $A \otimes A$-modules.
\end{proposition}

\begin{examples} If $M$ is an oriented closed manifold of dimension $n$ then $H^*(M)$ satisfies Poincaré duality in dimension $n$. The diagonal class $d_M$ for $H^*(M)$ is defined by
 
 \[ d_M= \sum_{i=1}^N (-1)^{deg(a_i)} a_i \otimes a^*_i \in (H^*(M) \otimes H^*(M) )^n,\]
 where $\{a_i\}_{1 \leq i \leq N}$ is a homogeneous basis of $H^*(M)$ and $\{a_i^*\}_{1 \leq i \leq N}$ is a basis of $H^*(M)$ characterised by the equations \[\langle a_i \cup a_j^*, [M] \rangle = \delta_{ij},\]
 and [M] is the fundamental class of $M$.
\end{examples}
A manifold $M$ is called \textit{formal} if $C^*(M)$ is quasi isomorphic to cohomology ring $H^*(M)$ with vanishing differential. If $M$ is formal then an $A$ in the above corollary is given by $H^*(M)$ itself.
\newline
For the rest of this section, we shall assume that our CDGA $A$ is a Poincaré duality algebra quasi-isomoprhic to singular cochain algebra of $M$ over the rationals. In preparation for modeling the string coproduct and the Goresky-hingston product, we first consider the Gysin homomorphism $H^{i}(\Theta) \xrightarrow{\tau} H^{i+n}(\Lambda \times \Lambda)$. A chain model for $\tau$, also described in [FT], is given by
\begin{multline}
CH_*(A)\otimes_{A} CH_*(A)=A\otimes T(s \bar{A})\otimes T(s \bar{A}) \xrightarrow{ \tau_{!} } A \otimes A \otimes T(s \bar{A}) \otimes T(s \bar{A})=CH_*(A)\otimes CH_*(A) \\
a_0 [ a_1, \cdots, a_l \otimes b_1, b_2, \cdots, b_m] \rightarrow \sum_{(a_0)} a_0' [a_1, \cdots, a_l ] \otimes a_0'' [b_1, b_2, \cdots, b_m]
\label{model2}
\end{multline}
where $ \delta(a_0)= \sum_{(a)}(a_0)' \otimes (a_0)''.$ \newline
We recall the construction here. We know that a chain model for the cochain map $ev_0^*:H_*(LM) \rightarrow H_*(M)$ which is induced by the evaluation map $ev_0: \Lambda \rightarrow M$ is given by inclusion $A \hookrightarrow A\otimes T(s \bar{A})$. We then have the following pullback diagrams.\newline
$\begin{CD}
 H^{*}(\Theta) @> {\tau} >> H^{*}( \Lambda \times \Lambda)\\
 @A ev^*_0 AA @A (ev_0 \times ev_0)^* AA\\
 H^{*}(M) @> D_{!} >> H^{*}(M \times M)
 \end{CD}
 $
 $\begin{CD}
 A\otimes T(s \bar{A})\otimes T(s \bar{A}) @>{ \tau_{!}} >> A\otimes T(s \bar{A}) \otimes A\otimes T(s \bar{A})\\
 @A AA @A AA\\
 A @> \delta >>A \times A
 \end{CD}
 $ \newline \\
 Here the Gysin map $D_{!}$ for the diagonal embedding $D : M \rightarrow M \times M$ is the Poincaré dual of the map $H^*(D)$.
\newline
We also need the following lemma. 
\begin{lemma}
The chain map
\begin{multline}
 f^h: CH_*(A) \otimes CH_*(A)= A \otimes T(s\bar{A})\otimes A \otimes T(s\bar{A}) \rightarrow A \otimes_{A^{2}} T(s\bar{A})=CH_*(A) \otimes_A CH_*(A) \\
 a_0[ a_1, \cdots, a_l] \otimes b_0[ b_1, b_2, \cdots, b_m] \rightarrow a_0 b_0 [a_1, \cdots, a_l, b_1, b_2, \cdots, b_m]
 \label{model1}
 \end{multline}
 is a chain model for the homomorphism $H^*(\Lambda \times \Lambda) \rightarrow H^*(\Theta)$ induced by the inclusion $ \Theta \hookrightarrow \Lambda \times \Lambda$.
 \end{lemma}
\begin{proof}
 We have the following pushout diagrams.
 \newline
$\begin{CD}
 H^{*}(\Theta) @< << H^{*}( \Lambda \times \Lambda)\\
 @A ev^*_0 AA @A (ev_0\times ev_0)^* AA\\
 H^{*}(M) @< \cup << H^{*}(M \times M)
 \end{CD}
 $
 $\begin{CD}
 A\otimes T(s \bar{A})\otimes T(s \bar{A}) @< f^h << A\otimes T(s \bar{A}) \otimes A\otimes T(s \bar{A})\\
 @A AA @A AA\\
 A @< << A \times A
 \end{CD}
 $ \newline \\
 Since the pushout of a chain model of a fibration is a chain model of the pullback of the fibration, and a chain model for the cup product $\cup$ given by the multiplication in $ A$, hence the lemma follows. 
\end{proof}
Lastly we consider the map
 $H^{i+n}(\Lambda \times \Lambda) \xrightarrow{ (ev^{\Lambda}_0 \times id)^* } H^{i+n}(\Lambda) $.
 \begin{lemma}\label {cut1}
 A chain model for the map $(ev^{\Lambda}_0 \times id)^*$ is given by
\[ A\otimes T(s \bar{A}) \otimes A\otimes T(s \bar{A}) \xrightarrow{} A\otimes T(s \bar{A})\]
 \begin{multline}
 a_0[a_1, a_2 \cdots, a_l] \otimes b_0[ b_1, b_2 \cdots, b_m]= \left\{ \begin{array}{l l}
 a_0 b_0[ b_1, b_2 \cdots, b_m], & \quad for \quad l=0 \\
 0, & \quad otherwise \\
\end{array} \right.
 \label{model3} 
\end{multline}
\end{lemma}

\begin{proof}
Let us denote by $i$ the inclusion $M \rightarrow \Lambda$.
First we observe that, a chain model for the identity map which factors as
 $ H^*(M) \xrightarrow{ev_0^*} H^*(\Lambda)\xrightarrow{i^*}H^*(M)$ is given by the identity map $A \rightarrow A$ factoring as $ A \hookrightarrow A \otimes T(s \bar{A}) \xrightarrow{i^h} A $, where 
 \[i^h( a_0[a_1, a_2 \cdots, a_l])= \left\{ \begin{array}{l l}
 a_0 & \quad for \quad l=0 \\
 0 & \quad otherwise \\
\end{array} \right. \]
So $i^h$ is a chain model for $i^*.$ \newline
By definition, the map $(ev^{\Lambda}_0)^*: H^*(\Lambda) \rightarrow H^*(\Lambda)$ factors as $ H^*(\Lambda) \xrightarrow{i^*} H^*(M) \xrightarrow{(ev_0)^*} H^*(\Lambda)$.
Therefore a chain model for $(ev^{\Lambda}_0)^*$ is given by \[A \otimes T(s \bar{A}) \rightarrow A \otimes T(s \bar{A})\]
 \begin{equation}\label{al evxid1}
a_0[a_1, a_2 \cdots, a_l] \rightarrow \left\{ \begin{array}{l l}
 a_0 & \quad for \quad l=0 \\
 0 & \quad otherwise \\
\end{array} \right.
 \end{equation}
 
 So a chain model for $(ev^{\Lambda}_0 \times id)^*$ is given by a chain map which gives the above map when restricted to the first component of $(A\otimes T(s \bar{A}) \otimes (A\otimes T(s \bar{A})$ and gives the identity map when restricted to the second component, i.e.
 \begin{equation}\label{al evxid2}
 (ev^{\Lambda}_0 \times id)^* (a_0[a_1, a_2 \cdots, a_l] \otimes \mathbf{1} )=
 \left\{ \begin{array}{l l}
 a_0 & \quad for \quad l=0 \\
 0 & \quad otherwise \\
\end{array} \right.
 \end{equation}
 
 and
 \begin{equation}
 (ev^{\Lambda}_0 \times id)^* ( \mathbf{1} \otimes b_0[ b_1, b_2 \cdots, b_m] )=
 b_0[ b_1, b_2, \cdots, b_m] 
 \end{equation}
where $\mathbf{1}$ is the unity in $A$. \newline
We also have the following commutative diagrams: \newline
$\begin{CD}
H^{*}( \Lambda \times \Lambda) @> {(ev^{\Lambda}_0 \times id)^* } >> H^{*}( \Lambda)\\
 @AA (ev_0\times ev_0)^* A @AA ev_0^* A\\
H^{*}(M \times M) @> \cup >> H^{*}( M)
 \end{CD}
 $
 $\begin{CD}
 A\otimes T(s \bar{A}) \otimes A\otimes T(s \bar{A})@>{ } >> A\otimes T(s \bar{A})\otimes T(s \bar{A}) \\
 @AA A @AA A\\
 A \times A@> >> A
 \end{CD}
 $ \newline \\
So a chain model for the map $(ev^{\Lambda}_0 \times id)^* $ should be a trivial map or of the form
 \begin{equation}\label{al evxid3}
 a_0[a_1, a_2 \cdots, a_l] \otimes b_0[ b_1, b_2 \cdots, b_m] \rightarrow
 a_0 b_0 [.,., \cdots,.] 
 \end{equation}

The lemma now follows from \ref{al evxid1}, \ref{al evxid2} and \ref{al evxid3}.
 \end{proof}

Thus from \ref{model1}, \ref{model2} and \ref{model3}, we obtain a chain model for $P^*_0$,
\[ P^h_0: CH_*(A)\otimes CH_*(A) \rightarrow CH_*(A),\]
\begin{equation}\label{p0h} a_0[ a_1, \cdots, a_l] \otimes b_0[ b_1, \cdots, b_m] \rightarrow \left\{ \begin{array}{l l} \sum_{(a_0 b_0)}(a_0 b_0)' (a_0 b_0)'' [b_1, b_2, \cdots, b_m] & for \quad l=0 \\
 0 & \quad otherwise \\
\end{array} \right.
\end{equation}
where $ \delta (a_0 b_0)= \sum(a_0 b_0)' \otimes (a_0 b_0)''$.
We recall that dual of the operation $P_0$ induces the string coproduct $\vee_0$. Thus the dual of the above model gives a chain model for the string coproduct $\vee_0$.

Interchanging the position of the maps $ev_0$ and $id$, and using similar argument we also obtain a chain model for $P^*_1$:
\[ P^h_1: CH_*(A)\otimes CH_*(A) \rightarrow CH_*(A),\]
\begin{equation}\label{p1h}a_0 [a_1, \cdots, a_l] \otimes b_0[ b_1, \cdots, b_m] \rightarrow \left\{ \begin{array}{l l}
 \sum_{(a_0 b_0)} (a_0 b_0)' (a_0 b_0)'' [a_1, a_2, \cdots, a_l] & for \quad m=0 \\
 0 & \quad otherwise \\
\end{array} \right. \end{equation}

We recall that dual of the operation $P_1$ induces the string coproduct $\vee_1$ defined in \ref{degenerate product}. Thus the dual of the above model gives a chain model for the string coproduct $\vee_1$.
\newline
Now, we verify that
 the map \[P^h: CH_*(A) \otimes CH_*(A) \rightarrow CH_*(A) \] given by
 \begin{equation}\label{Ph} a_0 [a_1, \cdots, a_l] \otimes b_0[ b_1, \cdots, b_m] \rightarrow \sum_{(a_0 b_0)}(a_0 b_0)' [a_1, a_2, \cdots, a_l, (a_0 b_0)'', b_1, b_2, \cdots, b_m] \end{equation} defines a chain homotopy between
 $\varGamma^{\tau}_0 \circ i ^*$ and $\varGamma^{\tau}_1\circ i^*$
 
 \[ P^h_1- P^h_0= d P^h + P^hd.\]

We observe that both maps $P^h_0$ and $P^h_1$ vanish on relative Hochschild chain complex $\tilde{CH}_*(A)$. Consequently, the homotopy $P^h$ becomes a chain map, and passing to the homology it gives rise to a product on relative Hochschild homology $\tilde{HH}_*(A)$.
\newline
Now, recall that the Goresky-Hingston product also arises in a similar way from the homotopy $P$ between $P_0$ and $P_1$ defined in \ref{Ph}. So one can expect $P^ h$ restricted to the relative Hochschild chain $\tilde{CH}_*(A)$ realises the Goresky-Hingston product. In \cite{NW19}, the authors prove that that is indeed the case. 
\section{Some properties of the string topology operations}
We have the following easy corollary of the algebraic interpretations of the string operations in the previous section. 
\begin{corollary} If $f:M_1 \rightarrow M_2$ be a rational homotopy equivalence between two closed simply connected manifolds then the induced homotopy equivalence of the loop spaces $\Lambda f: \Lambda (M_1) \rightarrow \Lambda (M_2)$ induces a ring isomorphism
\begin{equation}
(\Lambda f)_*: (H^*( \Lambda (M_1), M_1), \circledast) \cong ( H^{*}(\Lambda (M_2), M_2), \circledast).
\end{equation}
\end{corollary}
\begin{proof} In rational homotopy theory there is a bijection: 
\begin{multline} \{ \mbox{rational homotopy type of spaces}\} \\\longrightarrow \{ \mbox{isomorphism classes of minimal Sullivan algebras}\}.\end{multline}
By the theorem \ref{ls} there exist a CDGA $A$ satisfying Poincaré duality in dimension $n$ quasi-isomorphic to both $C^*(M_1)$ and $C^*(M_2)$. Thus the corollary follows.
\end{proof}
 The string coproducts $\vee_t$ defined in the equation \ref{degenerate product}, which are same for all $ t \in (0,1)$, are almost trivial. Owing to this fact the string coproducts are also known as the degenerate string coproducts.

\begin{proposition}\label{proofdegprod}
For any $x \in H_i(\Lambda),$ for any $t \in [0,1]$, we have
\[\vee_t (x) = \left\{ \begin{array}{l l} 0 & \quad for \quad i \neq n \\
 \chi(M).(x \odot \mathbf{1}) \otimes \mathbf{1} & \quad for \quad i= n \\
\end{array} \right. \] 

with $x \odot \mathbf{1} \in H_0(M)$ when $x \in H_n(M)$, where $ \chi(M)$ is the Euler characteristic of the manifold $M$, $\mathbf{1}$ is the generator of $H_0(M)$.
\end{proposition}
\begin{proof}

 Since the chain maps $P^h_0$ and $P^h_1$ in \ref{p0h} and \ref{p1h}respectively are homotopic, therefore they induce the same map on Hochschild homology. Combining their chain model we deduce that the induced map on homology $P^*_0= P^*_1$ is trivial on homology classes represented by chains with string length greater that one in one of the component, i.e., either $l\geq 1$ or $m \geq 1$. For chains with $l=0$ and $m=0$, $P^*_0= P^*_1$ is given by
 
 \[a_0 \otimes b_0 \rightarrow \sum_{(a_0 b_0)} (a_0 b_0)' (a_0 b_0)''.\]
 
Now a Hochschild chain with string length at most one can not represent a homology class of dimension more than $n$. It follows that, for $x \in H^i(\Lambda)$ and $y \in H^i(\Lambda)$, if the class $P^*_0(x, y) \in H^{i+j+n}(\Lambda)$ is non trivial then $i+j+n \leq n$, this implies $i=j=0$.
 
We deduce from the homology and cohomology pairing that the dual of $P^*_0$, $\vee_0$ is trivial on $ H_i(\Lambda)$ for $i \neq n$ and for $x \in H_i(\Lambda)$ with $i = n$
 \[ \vee_0(x)= \beta \ \mathbf{1} \otimes \mathbf{1} \in H_0(\Lambda) \times H_0(\Lambda)\]
 for some $\beta \in \mathbb{Q}$.
 For $a_0= b_0= \mathbf{1}$, the identify element of $A$, we have
 \[P_0( \mathbf{1} \otimes \mathbf{1} ) = \sum (-1)^{\mbox{deg} i }(a_i a^*_i), \] where $\{a_i\}$ is a homogeneous basis of $A$ and $\{a^*_i\}$ its dual basis. This implies
 \[P_0( \mathbf{1} \otimes \mathbf{1} ) = \chi(M) [M],\]
 where $[M]$ is an element representing the Poincaré dual of the fundamental class $[M]$. Since $P_0$ is dual of $\vee_0$, therefore $\vee_0([M])= \chi(M) \mathbf{1} \otimes \mathbf{1}$. It is known that $[M] \odot \mathbf{1}= \mathbf{1}$, thus
 \[\vee_0(x)= \chi(M) (x \odot \mathbf{1}) \otimes \mathbf{1}.\]
 \end{proof}
 
A closed geodesic is called prime if it is not a iterate of another closed geodesic. Let $M$ be a closed manifold all of whose geodesics are closed with same prime length. Example of such manifolds includes $n$-spheres for $n \geq 2$, complex and quaternionic projective spaces $ \mathbb{C}P^n, \mathbb{H}P^n$ for $n \geq 1$. Let $ \Sigma \subset \Lambda$ denote the submanifold consists of all the prime closed geodesics, it is diffeomorphic to the unit sphere bundle $SM$ of $M$. Let $\lambda_1$ denote the common index of prime closed geodesics.

 In \cite{GH09}, the authors showed that there is an isomorphism
\[h: H^{*}(\Sigma) \rightarrow H^{* + \lambda_1}(\Lambda, \Lambda_0),\]
and proved the following proposion. 
\begin{proposition}\label{ghi}
Let $M$ be manifold all of whose geodesics are closed with same prime length. Then the ring $(H^*(\Lambda,\Lambda_0), \circledast)$ is finitely generated and it is generated by the images of the generators of $H^*(\Sigma)$ by the map $h$.
\end{proposition}

 A manifold $M$ is called monogenic if its singular cohomology ring $H^*(M)$ with the cup product is generated by a single element. In this case the ring $H^*(M)$ is isomorphic to the truncated polynomial ring $\mathbb{Q}[x]/{x^{r+1}=0}$ for some $r$, where $x$ is the generator. The examples of such manifold include spheres, complex projective spaces.

\begin{proposition}
Let $M$ be a simply-connected monogenic manifold, then the ring $(H^*(\Lambda(M), $ $\Lambda_0), \circledast)$ is finitely generated. The number of generator is same as the dimension of the cohomology group of the sphere bundle of $M$.
\end{proposition}
\begin{proof}
Let $H^*(M)=\mathbb{Q}[x]/{x^{r+1}=0} $, for some $r$ and for some $x$ with $\mbox{dim}(x)=m$.
We know that the cohomology ring of the complex projective space $\mathbb{C} P^r$ is $\mathbb{Q}[x]/{x^{r+1}=0}$, where $x$ is the generator of degree $2$. Since $M$ is simply connected monogenetic so it is formal. Consequently, using the algebraic model for the G-H product \ref{Ph}, we deduce that the ring $(H^*(\Lambda(M), \Lambda_0), \circledast)$ is completely determined by $H^*(M)$, and $(H^*(\Lambda(M), \Lambda_0), \circledast)$ is isomorphic to $(H^*(\Lambda(\mathbb{C} P^r), \Lambda_0), \circledast)$. Now we know from \cite{GH09} that the ring $(H^*(\Lambda(\mathbb{C} P^r), \Lambda_0), \circledast)$ is finitely generated for any $r \in \mathbb{N}$. The statement on generators follows from Proposition \ref{ghi}. 
\end{proof}

\subsection{The G-H product for the spheres}The ring $(H^*(\Lambda(S^n), S^n), \circledast)$ was computed in \cite{GH09} using purely geometric method. Here we compute the same ring using the algebraic model for the product. 
\newline
\underline{\textbf{$(H^*(\Lambda(S^n), S^n), \circledast)$ for $n$ odd}} \newline
For $n$ odd, it is known that
\[ H^m(SS^n)= \left\{ \begin{array}{l l}
 \mathbb{Q} & \quad \mbox{for} \quad m=0, n-1, n, 2n-1 \\
 0 & \quad \mbox{otherwise} \\
\end{array} \right. \]
and
\[ H^k(\Lambda(S^n))= \left\{ \begin{array}{l l}
 \mathbb{Q} & \quad \mbox{for} \quad k=m(n-1), m(n-1)+n, m\geq 0 \\
 0 & \quad \mbox{otherwise} \\
\end{array} \right. \]

Since $\Lambda= n-1$ for $S^n$, therefore the image of generators of the cohomology group $H^*(S(S^n))$ by $h$ are cohomology classes in dimensions $n-1, 2n-2, 2n-1, 3n-2$. Since the dimension of the cohomology groups $H^k(\Lambda(S^n), S^n)$ in each dimension is at most one, therefore by proposition \ref{ghi} they represent the generators of the ring $(H^*(\Lambda(S^n), S^n), \circledast)$.

 Now, let $\mathbf{1}$ be the generator of $H^0(S^n)$ and $x_1$ be the generator of $H^n(S^n)$. Then it is easy to check that the elements $\omega= \mathbf{1}[x_1]$, $X= \mathbf{1}[x_1, x_1]$, $ Y= x_1[x_1]$, $Z= x_1[x_1, x_1]$ define non-trivial homology classes in the dimensions $n-1, 2n-2, 2n-1, 3n-2$ respectively. Thus these four elements are the generators of the ring $ (\tilde{HH}_*(H^*(M)), \circledast)$. Since the coproduct of $\mathbf{1}$, $ \delta({\mathbf{1}})= (\mathbf{1} \otimes \mathbf{1}) \cup d_{S^n}= \mathbf{1} \otimes x_1 $, where $d_{S^n}= \mathbf{1} \otimes x_1$
 be the diagonal class of $S^n$, we check that $ \omega, X, Y, Z$ satisfy the following relations $$ X^2= \omega^{3}, \quad X \circledast Z= \omega^{3}\circledast Y, \quad X \circledast Y= Y \circledast X = Z \circledast \omega= \omega \circledast Z$$
 $$Z \circledast Z= Y \circledast Y = Y \circledast Z=0,$$
 Setting $ \omega= T^2, X=T^3, Y=U \otimes T^2$ and $Z=U \otimes T^3$, we have \[H^*(\Lambda(S^n), S^n)\cong \wedge (U) \otimes \mathbb{Q}[T]_{\geq 2},\] where $\mathbb{Q}[T]_{\geq 2}$ is denotes the ideal $(T^2)$ in the ring of polynomials. Here we note that the computation produces the same ring structure as in \cite{GH09}.
 \newline
\underline{\textbf{$(H^*(\Lambda(S^n), S^n), \circledast)$ for $n$ is even}}
\newline
For $n$ even, it is known that
\[ H^m(SS^n)= \left\{ \begin{array}{l l}
 \mathbb{Q} & \quad for \quad m=0, 2n-1 \\
 0 & \quad otherwise \\
\end{array} \right. \]
and
\[ H^k(\Lambda(S^n))= \left\{ \begin{array}{l l}
 \mathbb{Q} & \quad for \quad k=0, (2m+1)(n-1)+1, (2m+1)(n-1), m\geq 0 \\
 0 & \quad otherwise \\
\end{array} \right. \]

Similar to above case for odd $n$ we can show that the elements $\omega= \mathbf{1}[x_1]$, $Z= x_1[x_1, x_1]$ are the generators of the ring $ (HH_*(H^*(M), H^*(M)), \circledast)$ and they satisfy
$$ Z \circledast Z= 0$$
Therefore, $H^*(\Lambda(S^n), S^n)\cong \wedge (Z) \otimes \mathbb{Q}[w]$. \newline
\subsection{The G-H product for the complex projective spaces} The following facts are known. 
\[ H^m(S \mathbb{C} P^n)= \left\{ \begin{array}{l l}
 \mathbb{Q} & \quad for \quad m=2i, i=0,1, \cdots, n-1, \\
 \mathbb{Q} & \quad for \quad m=2i+2n-1, \quad i=1, \cdots, n, \\
 0 & \quad otherwise \\
\end{array} \right. \]
and 
\[ H^k(\Lambda(\mathbb{C} P^n))= \mathbb{Q}\quad \mbox{for all} \quad k.\]

Since $\Lambda=1$ for $\mathbb{C} P^n$, Since $\Lambda= n-1$ for $S^n$, therefore the image of the generators of cohomology group $H^*(S(S^n))$ by $h$ are cohomology classes in dimensions $2i-1, 2i+2n$ for $i= 0, 1,\cdots, n$. Since the dimension of the cohomology groups $H^*(\Lambda(\mathbb{C} P^n), \mathbb{C} P^n)$ in any dimension is at most one, therefore by proposition \ref{ghi} they represent the generators of the ring $(H^*(\Lambda(\mathbb{C} P^n), \mathbb{C} P^n), \circledast)$.
\newline
Let $\beta$ be a generator of the ring $H^*(\mathbb{C} P^n)$, then the elements $ \mathbf{1}[\beta^i]$, for $i\neq 4 m$, are some generators of the ring $(H^*(\Lambda(\mathbb{C} P^n), \mathbb{C} P^n), \circledast)$, but the remaining generators are not so easy to identify.
 
\section{Optimal index growths of closed geodesics}\label{applications}
In this section we discuss the relations between string topology operations and index growth of closed geodesics and its implications on existence of closed geodesics. 
We begin by giving a brief account of the known results on the problem. 
\newline
The existence of at least one closed geodesic on any compact manifold was shown by Lyusternik and Fet \cite{LF51}, extending the works of Birkhoff \cite{Br66}. If the fundamental group is non-trivial 
then one merely needs to look at the closed curve representations with minimal length. For simply 
connected space one uses Birkhoff's minimax argument.

The existence of infinitely many geometrically distinct closed geodesics on surfaces is shown using a 
similar minimax argument as of Birkhoff's. In \cite{HR89}, Rademacher showed that a simply connected compact 
manifold with a generic Riemannian metric admits infinitely many geometrically distinct closed 
geodesics.

The main obstacle in finding geometrically distinct closed geodesics is that the iterates of closed 
geodesics may also contribute to the homology of the free loop space. The contributions are often 
controlled by their index and nullity.

The following theorem due to R. Bott \cite{BT56} gives us estimates of indices of iterates.

\begin{theorem}
Let $ \gamma$ be a closed geodesic of a $n$-dimensional compact manifold $M$ and let $
\lambda_m$ and $
\nu_m$ are the index and the nullity of the $m$-fold iterate $\gamma^m$ respectively. Then $ \nu_m 
\leq 
2(n-1)$ for all $m$ and 
\begin{equation}\label{indexgrowth}
\begin{split} |\lambda_m -m \lambda_1 | \leq (m-1)(n-1), \\ |\lambda_m + \nu_m -m (\lambda_1 + 
\nu_1)| \leq 
(m-1)(n-1),
\end{split}
\end{equation}

Furthermore, if $\gamma$ and $\gamma^2$ are nondegenerate, then for any $k \geq 0$, there 
exists a constant $C_k$ such that for all $m \geq C_k$, 

\begin{equation}\label{indexgrowth2}
\begin{split} |\lambda_m -m \lambda| \leq (m-1)(n-1)-k, \\ |\lambda_m + \nu_m -m (\lambda_1 + \nu_1)| \leq 
(m-1)(n-1)-k.
\end{split}
\end{equation}

\end{theorem}

The contributions of a geodesic to local level homologies are better 
understood than homology of the whole loop space. We recall below some of the well known results 
on this.

 Let $ \gamma$ be a closed geodesic and $\Gamma^- \ra \bar{\gamma}$ be the negative bundle whose fibers are the negative eigenspaces of Hessians of $F$.
 Then the change in topology of free loop space while crossing a critical orbit can be
 expressed homologically as follows:

\begin{theorem}[\cite{BT54, RP63}] Let $\gamma$ be nondegenerate closed geodesic on a compact manifold $M$ of 
index $
\lambda$. Let $G= \z$ when the negative bundle $\Gamma^- \ra \bar{\gamma}$ is oriented, and $G= 
\z_2$ otherwise. Then the local level homology groups are given by
\[ H_i(\La^{<a} \cup \bar{\gamma}, \Lambda^{<a})\cong \left\{ \begin{array}{l l}
 G & \quad for \quad i=\lambda, \lambda + 1 \\
 0 & \quad otherwise. \\
\end{array} \right. \]
An analogous statement holds for cohomology. 
\end{theorem}

 For degenerate case, we do not unfortunately have a definitive answer as the above. However the following result due to Gromoll and Meyer gives a very useful estimate.

\begin{theorem}[\cite{GM69}] \label{GrM}Let $\gamma$ be an isolated (possibly degenerate) closed geodesic with 
index $\lambda$ and nullity $\nu$. If a local level homology group $H_i(\La^{<a} \cup \bar{\gamma}, 
\La^{<a}) \neq 0$ then $i \in [\lambda, \lambda + \nu +1]$. The same holds for cohomology groups. 
\end{theorem}

A consequence of the above theorem is the following celebrated theorem of Gromoll and Meyer 
\cite{GM69}.
 
\begin{theorem} Let $M$ be a compact simply connected Riemannian manifold such that the 
sequence of Betti numbers $\beta_i(\Lambda(M))$ is unbounded (with some field of coefficients). 
Then $M$ admits infinitely many geometrically distinct closed geodesics.
\end{theorem}

In \cite{PMS77}, Vigué-Poirier and Sullivan showed that the hypothesis of the above theorem is 
satisfied when the manifold is sufficiently complicated, more precisely: 

\begin{theorem}\label{pois}
For a simply connected compact manifold $M$, then the cohomology algebra of $M$ with rational 
coefficients has at least two generators if and only if the Betti numbers of $\Lambda(M)$ with rational 
coefficients are unbounded.
\end{theorem}
This is in particular true for globally symmetric spaces of rank $ > 1$ \cite{Zi77}.
\newline
Among the spaces which do not satisfy the assumptions of the Gromoll-Meyer theorem, are the 
symmetric spaces of rank $1$, which includes our most familiar spaces $n$-spheres. Below we state 
two theorems of N. Hingston \cite{Hi93, Hi97}, which can be applied to show that any Riemannian 
$2$-sphere admits infinitely many closed geodesics. 

\begin{theorem}\label{Hingston2}
Let $ \gamma$ be a closed geodesic on an $n$-dimensional compact manifold $M$ with 
$F(\gamma)=a$. With the notations as above, assume that $H_{\lambda_1 + \eta_1 + 1}( \La^{< a} 
\cup \gamma, \La^{< a}) \neq 0$ and $\lambda_m + \nu_m = m(\lambda_1 + \eta_1)- (n-1)(m-1)$ for 
all $m \geq 1$. Then $M$ has infinitely many closed geodesics.
\end{theorem}

 \begin{theorem}\label{Hingston1}
Let $ \gamma$ be a closed geodesic on an $n$-dimensional compact manifold $M$ with 
$F(\gamma)=a$. With the notations as above, assume that 
$ H^{\lambda_1 }( \La^{< a} \cup \gamma, \La^{< a}) \neq 0$ and $\lambda_m = m \lambda_1 + (n-1)
(m-1)$ for all $m \geq 1$. Then $M$ has infinitely many closed geodesics.
\end{theorem}

To the author's knowledge, the closed geodesic problem remains largely open for symmetric spaces 
of rank $1$. We refer readers to the survey articles \cite{BM10, AO14} for various other results and 
conjectures about the problem. 
\subsection{Level products}

In \cite{GH09}, the authors showed that the C-S product extends (compatible with the product on 
homology of $\La$) naturally to a product on level homologies as follows: 

\begin{equation} \label{ghext}
H_i(\Lambda^{\leq a}, \Lambda^{< a}) \times H_j(\Lambda^{\leq b}, \Lambda^{< b}) \ra H_{i+j-n}
(\Lambda^{\leq a+b}, \Lambda^{< a+b}).
\end{equation}
Similarly, the G-H product extends to level cohomologies as follows: 
\begin{equation}
H^i(\Lambda^{\leq a}, \Lambda^{< a}) \times H^j(\Lambda^{\leq b}, \Lambda^{< b}) \ra H^{i+j+n-1}
(\Lambda^{\leq a+b}, \Lambda^{< a+b}).
\end{equation}

 Let $A, B \subset \Lambda$ we write $ A \times_M B= (A \times B) \cap\Theta$. Let $\phi_{min}: 
(\Theta- \Lambda_0) \ra \Lambda$ defined by 
 $ \phi_{min}(\alpha, \beta)= \phi_s( \alpha, \beta)$ with $s=\frac{F (\alpha) } {F(\alpha) + F(\beta)}$, 
which extends continuously to $\Lambda_0$ giving a map $\phi_{min}: \Theta \ra \Lambda$, and 
define $A \bullet B= \Phi_{min}(A \times_M B)$. In particular, for any $\gamma \in \La$ and $m \geq 
1$, $ \overline{\gamma^m} \bullet \bar{\gamma} = \overline{\gamma^{m+1}}. $
 
The following is a restatement of Proposition 10.2 in \cite{GH09}. 
\begin{proposition} \label{localghprod}
 Let $A$ and $B$ be any closed subsets of $\La^{\leq a}- \Lambda_0$ and $\La^{\leq b}- 
\Lambda_0$ respectively, for $0 \leq a, b \leq \infty$. Then the G-H product extends to a product 
\[ H^i( \La^{\leq a}, \La^{\leq a} -A ) \times H^i( \La^{\leq b}, \La^{\leq b } - B ) \xrightarrow{\circledast} 
H^{i+j-n+1}( \La^{\leq a +b}, \La^{\leq a +b} -A \bullet B ).\] 
\end{proposition}
For an isolated closed geodesic $ \gamma$ with $F(\gamma)=a$, using lemma 4.2.3 in \cite{Kl78}, 
we have \newline
\[ H^{*}( \La^{< a} \cup \bar{\gamma}, \La^{< a} ) \cong H^{*}( \La^{\leq a}, \La^{\leq a} -
\bar{\gamma} ). \]

Assume that all the orbits of iterates of $\gamma$, $\bar{\gamma^m}$ are isolated. Then, taking $A= 
\overline{\gamma^m}$ and $ B= \bar{\gamma} $ in the above proposition gives 
\begin{equation} \label{locextgh}
 H^{i}( \La^{< ma} \cup \overline{\gamma^{m}}, \La^{<m a} ) \times H^{j}( \La^{< a} \cup 
\bar{\gamma}, \La^{< a} ) \xrightarrow{\circledast} H^{i+j+n-1}( \La^{< (m+1) a} \cup 
\overline{\gamma^{m+1}}, \La^{< (m+1) a} ).
\end{equation}

Similarly, it follows from Proposition 5.3 in \cite{GH09}( as in Theorem 11.3 for nondegenerate case), 
that the C-S product extends to 
\begin{equation} \label{locextcs} H_{i}( \La^{< ma} \cup \overline{\gamma^{m}}, \La^{<m a} ) \times 
H_{j}( \La^{< a} \cup \bar{\gamma}, \La^{< a} ) \xrightarrow{\bullet} H_{i+j-n}( \La^{< (m+1) a} \cup 
\overline{\gamma^{m+1}}, \La^{< (m+1) a} ). \end{equation}

Now we prove two theorems mentioned in the introduction.

 \begin{theorem} \label{second}
Let $\gamma$ be a closed geodesic, and assume that the orbits of all its iterates are isolated. If $
\gamma$ admits a nonnilpotent local level homology class with respect to the C-S product then $ 
\mathrm{index}(\gamma^m)+ \mathrm{nullity} (\gamma^m)= m \ (\mathrm{index} (\gamma)+ 
\mathrm{nullity}(\gamma))- (m-1)(n-1)$. 
\end{theorem}

\begin{proof}
 Let $x \in H_{j}
( \Lambda^{<a} \cup \gamma, \Lambda^{< a})$ be a nonnilpotent level homology class of $
\gamma$, where $a= F(\gamma)$. By \ref{locextcs}, this means

 \[ x^{\bullet m} \in H_{mj-n(m-1)}( \La^{< (m+1) a} \cup \overline{\gamma^{m}}, \La^{< (m) a}) \] is a 
non-trivial class for all $m \geq 1$. Denote by $\lambda_m= \mathrm{index}(\gamma), \nu_m= 
\mathrm{nullity}(\gamma)$. It follows then from the Gromoll and Meyer theorem \ref{GrM} that

\begin{equation} \label{average0} \lambda_m \leq mj - n(m-1) \leq \lambda_m+ \nu_m+ 1 \ \mbox{for 
all} \ m \geq 1, \end{equation}
and in particular,
\begin{equation} \label{average1} \lambda_1 \leq j \leq \lambda_1+ \nu_1+ 1. \end{equation}

Since $\nu_m \leq 2n-1$ for all $m$, it follows that the average index of $\gamma$ 

\begin{equation}\label{average2} \bar{\lambda} = \lim_{m \ra \infty} \frac{ \lambda_{ m}}{m} = j-n. 
\end{equation}

On the other hand, we have the following from Bott's index estimate \ref{indexgrowth}
\[m (\lambda_1 + \nu_1 ) -(m-1) (n-1) \leq \lambda_m + \nu_m \ \mathrm{for \ all} \ m \geq 1,\]
 (as index and nullity are natural numbers). This implies 
\begin{equation} \label{average3} \lambda_1 + \nu_1 -n+1 \leq \bar{\lambda}. 
\end{equation}

Combining the above equations \ref{average1}, \ref{average2} and \ref{average3}, we get \[ j= 
\lambda_1 + \nu_1 +1.\]

So Equation \ref{average1} can be rewritten as 
\begin{equation} \label{average4}
\lambda_m \leq m (\lambda_1 + \nu_1) - (m-1)(n-1) +1 \leq \lambda_m + \nu_m +1.
\end{equation}

 Now we claim that
 \[ \lambda_m + \nu_m = m (\lambda_1 + \nu_1) - (m-1)(n-1) \ \mathrm{for \ all} \ m \geq 1.\]
 
 Suppose our claim is false, then it follows again from \ref{indexgrowth} that for some positive 
integers $k$ and $C$ we have 
 \[ \lambda_k + \nu_k = k (\lambda_1 + \nu_1) - (k-1)(n-1) +C.\]
 
 Then for any positive integer $r$, we have 
 
\[ \lambda_{rk} + \nu_{rk} \geq r(k (\lambda_1 + \nu_1) - (k-1)(n-1) +C) -(r-1) (n-1) \]
\[ \implies \lambda_{rk} \geq rk (\lambda_1 + \nu_1) - (rk-1)(n-1) +rC - \nu_{rk}. \]

This contradicts the first inequality of Equation \ref{average4} for large $r$. This proves our claim, 
hence completes the proof of the theorem.

\end{proof}

\begin{theorem} \label{third}
Let $\gamma$ be a closed geodesic, and assume that the orbits of all its iterates are isolated. If $
\gamma$ admits a nonnilpotent local level cohomology class $x$ with respect to the G-H product then $
\dim(x)=\mathrm{index}(\gamma)$ and $\mathrm{index}(\gamma^m)= m \ \mathrm{index}(\gamma) + 
(m-1)(n-1)$.\end{theorem}

\begin{proof}
 Let $ x \in H^{j}( \Lambda^{<a} \cup \bar{\gamma}, \Lambda^{< a})$ be a nonnilpotent level 
cohomology class produced by $\gamma$, where 
$a=F(\gamma)$. By \ref{locextcs}, this means

 \[ x^{\circledast m} \in H^ {mj + (m-1)(n-1)}( \La^{< (m+1) a} \cup \overline{\gamma^{m}}, \La^{< (m) 
a}) \] is a non-trivial class for all $m \geq 1$. It follows then from the theorem of Gromoll and Meyer 
\ref{GrM} that
\begin{equation} \label{bott1} \lambda_m \leq mj + (m-1)(n-1) \leq \lambda_m+ \nu_m+ 1 \ \mbox{for 
all} \ m \geq 1, \end{equation} 

and in particular,
\begin{equation} \label{average7} \lambda_1 \leq j \leq \lambda_1+ \nu_1+ 1. \end{equation}

 It follows that the average index of $\gamma$ 

\begin{equation}\label{average5} 
\bar{\lambda} = j +n-1
\end{equation}

On the other hand, we have the following from Bott's index estimate \ref{indexgrowth} 

\[ \lambda_m \leq m \lambda_1 + (m-1) (n-1) \ \mathrm{for \ all} \ m \geq 1,\]
which implies 
\begin{equation} \label{average6} \bar{\lambda} \leq \lambda_1 + n-1. 
\end{equation}

Combining the equations \ref{average7}, \ref{average5} and \ref{average6}, we get 
\[ j=\lambda_1= \mathrm{index}(\gamma).\]

Using a very similar argument as in the last theorem, we can then prove that $\gamma$ has maximal index growth, i.e., \[ \hspace{2mm} 
\lambda_m = 
m \lambda_1 + (m-1)(n-1) \hspace{2mm} \mbox{for all }m \geq 1. \]
\end{proof}

We have the following corollary. 
\begin{corollary} Let $M$ be a compact Riemannian manifold with $F: \Lambda \ra \r$ defined as 
above. If there is a nonnilpotent level homology class with respect to the C-S product then $M$ has 
infinitely many closed geodesics. An analogous statement holds true for cohomology with the G-H 
product. 

\end{corollary}

\begin{proof}
Let $ x \in H_*(\La^{\leq a}, \La^{< a})$ be a nonnilpotent level homology class, for some $a \geq 0$. 
Suppose that there are only finitely many geodesics $\gamma_1, \gamma_2, \cdots, \gamma_{r} $ 
such that $F(\gamma_i)=a$ for $i=1, 2, \cdots, r$, so they are obviously isolated. Consequently,
\begin{equation}\label{pencor} H_*(\Lambda^{\leq a}, \Lambda^{<a})\cong \oplus_{i=1}^r 
H_*(\Lambda^{<a} \cup \bar{\gamma_i}, \Lambda^{<a}). \end{equation}
Write $x= \sum_{i=1}^rx_i$, then there must exist a nonnilpotent class 
\[ x_k \in H_*(\Lambda^{<a} \cup \bar{\gamma_k}, \Lambda^{<a}),\]
for some $\gamma_k$, $1 \leq k \leq r$. Then by Theorem \ref{second}, $\gamma_k$ satisfies the 
hypothesis of Theorem \ref{Hingston2}, therefore, $M$ has infinitely many closed geodesics.
It is easy to see that if there are infinitely many geodesics at level $a$ then $M$ has infinitely many 
closed geodesics. 
\newline
A very similar argument using Theorem \ref{third} for cohomology can be made to show the 
existence of a geodesic satisfying the hypothesis of Theorem \ref{Hingston1}. Therefore, $M$ has 
infinitely many closed geodesics. This completes the proof. 
\end{proof}
The above corollary, in particular, gives a proof of two theorems, whose statement appeared in \cite{GH09}, (Theorems 12.6 and 12.7) but the proofs are not found anywhere in the literature. Here we note that our proof will require the additional assumption that the closed geodesic along with its iterates are isolated. To the author's understanding this assumption is necessary. 
\newline
The following corollary gives a simpler proof of Theorems 7.3 and 11.3 in \cite{GH09} put 
together.

\begin{corollary} Let $M$ be a compact manifold. Then for a generic Riemannian metric on $M$, with 
$F: \Lambda \ra \r$ defined as above, for every $a \geq 0$, every local level homology class $ x \in 
H_*(\La^{< a} \cup \bar{\gamma}, \La^{< a})$ is level-nilpotent. The analogous statement holds true for 
cohomology. 
\end{corollary}
\begin{proof}
For a generic Riemannian metric (also known as bumpy metric) on a manifold $M$, all closed geodesics 
are nondegenerate \cite{Kl78}. By Bott's index estimate \ref{indexgrowth2}, the indices of iterates of any closed 
geodesic can not be the maximum possible. Nondegenerate closed geodesics are obviously isolated. Then by the 
above theorem, any local level homology or cohomology class of any geodesic is nilpotent. Then the 
proof is concluded using \ref{pencor}.
\end{proof}

\section{Acknowledgement}
The author was supported by the National Board of Higher Mathematics (No.2018/R\&D-II/ 8872) 
during the academic year 2018--2019.

\bibliographystyle{amsplain}
\bibliography{all}
\end{document}